\documentclass[11pt,reqno]{amsart}
\usepackage[all]{xy}
\usepackage{amssymb}
\usepackage{amsthm}
\usepackage[normalem]{ulem}
\usepackage{amsmath,mathtools}
\usepackage{amscd,enumitem}
\usepackage{verbatim}
\usepackage{eurosym}
\usepackage{float}
\usepackage{color}
\usepackage{dcolumn}
\usepackage[mathscr]{eucal}
\usepackage[all]{xy}
\usepackage{bbm}
\usepackage[textheight=8.5in, textwidth=6.7in]{geometry}
\newtheorem*{conj*}{Conjecture}
\newtheorem{theorem}{Theorem}[section]

\theoremstyle{definition}
\newtheorem*{remark}{Remark}
\theoremstyle{plain}

\newtheorem{corollary}[theorem]{Corollary}

\newcommand{\Z}{\mathbb{Z}}

\newcommand{\N}{\mathbb{N}}

\newcommand{\SL}{\operatorname{SL}}
\newcommand{\C}{\mathbb{C}}
\newcommand{\im}[1]{\text{Im}\(#1\)}

\renewcommand{\pmod}[1]{\,\,({\rm mod}\,\,{#1})}

\DeclareMathOperator{\Arg}{Arg}
\numberwithin{equation}{section}

\newtheoremstyle{example}
  {\topsep}   
  {\topsep}   
  {\normalfont}  
  {0pt}       
  {\bfseries} 
  {.}         
  {5pt plus 1pt minus 1pt} 
  {}          
\theoremstyle{example}

\def\({\left(}
\def\){\right)}

\usepackage{centernot}

\begin{document}
\title{Exact formula for cubic partitions}
\author{Lukas Mauth}
\address{Department of Mathematics and Computer Science\\Division of Mathematics\\University of Cologne\\ Weyertal 86-90 \\ 50931 Cologne \\Germany}
\email{lmauth@uni-koeln.de}

\keywords{Circle Method, $\eta$-function, partitions}

\begin{abstract} 
	We obtain an exact formula for the cubic partition function and prove a conjecture by Banerjee, Paule, Radu and Zeng.
\end{abstract}

\maketitle
\section{Introduction and statement of results}

Let $n$ be a non-negative integer. A non-increasing finite sequence of positive integers that sums to $n$ is called a {\it partition} of $n.$ We denote by $p(n)$ the number of partitions of $n,$ which can be defined by the coefficients of the following $q$-series \cite{Andrews}

\begin{equation}\label{PartitonGeneratingFunction}
\sum_{n=0}^{\infty} p(n)q^n = \prod_{n=1}^{\infty} \frac{1}{1-q^n}
\end{equation}

\noindent
 The function $p(n)$ is one of the most well-studied functions troughout number theory and satisfies numerous remarkable identities. One example is the following identity discovered by S. Ramanujan \cite{Andrews}, \cite{RamanujanCollectedPapers}.

\begin{equation*}
 \sum_{n=0}^{\infty} p(5n+4)q^n = 5 \prod_{n=1}^{\infty} \frac{(1-q^{5n})^5}{(1-q^n)^6}.
\end{equation*}

\noindent
An immediate consequence is one of Ramanujan's celebrated congruences which asserts that for any non-negative integer

\begin{equation*}
p(5n+4) \equiv 0 \pmod 5
\end{equation*}

Similar to \eqref{PartitonGeneratingFunction} we define for non-negative integers the cubic partition function $a(n)$ by

\begin{equation}\label{CubicPartitionGeneratingFunction}
	\sum_{n=0}^{\infty} a(n)q^n = \prod_{n=1}^{\infty} \frac{1}{(1-q^n)(1-q^{2n})}.
\end{equation}

\noindent
H. Chan proved that the function $a(n)$ satisfies a identity similar to the one by Ramanujan given above

$$\sum_{n=0}^{\infty} a(3n+2)q^n = 3 \prod_{n=1}^{\infty} \frac{(1-q^3)^3(1-q^6)^3}{(1-q)^4(1-q^2)^4},$$

$$a(3n +2) \equiv 0 \pmod 3,$$

\noindent
see \cite{Chan}. His result is closely related to Ramanujan's continued cubic fraction, see \cite{ChenLin}. For an introduction to Ramanujan's cubic continued fraction, see \cite{AndrewsBerndt},\cite{Berndt}. Furthermore, from the generating function \eqref{CubicPartitionGeneratingFunction} it is obvious that $a(n)$ is the number of partition pairs $(\lambda, \mu)$ such that $|\lambda| + |\mu| = n$ and $\mu$ consists of only even numbers. These connections gave rise to the term "cubic partitions".

The first $30$ values of $a(n)$ are: $1, 1, 3, 4, 9, 12, 23, 31, 54, 73, 118, 159, 246, 329, 489, 651, 940, 1242, 1751,$ 
$2298, 3177, 4142, 5630, 7293, 9776, 12584, 16659, 21320, 27922, 35532, 46092.$

There is another combinatorial interpretation of the cubic partition function. Suppose you have a real polynomial equation of degree $n.$  Then, $a(n)$ counts the number of possibilities for the roots to be real and unequal, real and equal (in various combinations), or simple or multiple complex conjugates. For example, for a real cubic polynomial, one can have $3$ real distinct roots, $3$ real roots and two are equal, $3$ real roots and $1$ real root with a complex conjugate pair of $2$ roots. Hence, $a(3) = 4$.

Another natural question about $p(n)$ is how fast it grows. In the beginnig of the $20$th century G.H. Hardy and Ramanujan showed the following asymptotics with their celebrated Circle method \cite{HardyRamanujan}

\begin{equation} \label{PartitionAsymptotics}
	p(n) \sim \frac{1}{4n\sqrt{3}}\exp\bigg(\pi \sqrt{\frac{2n}{3}}\bigg), \quad n \rightarrow \infty.
\end{equation}

\noindent
Hardy and Ramanujan made extensive use of the fact that the generating function \eqref{PartitonGeneratingFunction} is almost a modular form. More precisely, 

\begin{equation}
\sum_{n=0}^{\infty} p(n)q^n = \prod_{n=1}^{\infty} \frac{1}{1-q^n} = \frac{q^{\frac{1}{24}}}{\eta(\tau)}
\end{equation},

\noindent
where $\eta(z)$ is the Dedekind eta function defined on the upper half-plane $\mathbb{H} = \{\tau \in \C \mid \im \tau > 0\}$ by 

$$\eta(\tau) = q^{\frac{1}{24}} \prod_{n=1}^{\infty} (1-q^n), \quad q=\exp(2\pi i \tau).$$ 

\noindent
The function $\eta(\tau)$ is a modular form of weight $\frac{1}{2}$ and satisfies the transformation law, for $\gamma \in \SL_2(\Z),$

\begin{equation*}
\eta\bigg(\frac{a\tau +b}{c\tau +d}\bigg) = \varepsilon(a,b,c,d) \sqrt{-i(c\tau + d)}\eta(\tau), \quad \gamma = \begin{pmatrix}
a &b \\
c & d
\end{pmatrix}, 
\end{equation*}

\begin{equation}\label{DedekindEtaFunction}
\varepsilon(a,b,c,d) := \begin{cases}
e{\frac{\pi i b}{12}} & \text{if } c=0, d=1, \\
e^{\pi i \big(\frac{a+d}{12}-s(d,c)\big)} & \text{if } c>0.
\end{cases}
\end{equation}

\noindent
Here $s(h,k)$ is the Dedekind sum 

\begin{equation*}
s(h,k) := \sum_{n=1}^{k-1} \frac{n}{k} \bigg(\frac{hn}{k}- \bigg\lfloor\frac{hn}{k}\bigg\rfloor - \frac{1}{2} \bigg).
\end{equation*}

\noindent
A proof can be found in \cite{Apostol}. A couple years later Rademacher \cite{RademacherExact} perfected the circle method and obtained the exact formula

\begin{equation}\label{Rademacher}
p(n) = \frac{1}{\pi\sqrt{2}} \sum_{k=1}^{\infty} A_k(n) \sqrt{k}\cdot\frac{d}{dn}\left(\frac{\sinh\bigg[\frac{\pi}{k}\sqrt{\frac{2}{3}\left(n-\frac{1}{24}\right)}\bigg]}{\sqrt{n-\frac{1}{24}}}\right),
\end{equation}

\noindent
where $A_k(n)$ is the Kloosterman sum given by

\begin{equation*}
A_k(n) := \sum_{\substack{0\leq m < k \\ \gcd(m,k) = 1}} \exp\left(\pi i \left(s(m,k)-\frac{2nm}{k}\right)\right).
\end{equation*}

\noindent
Note that the series \eqref{Rademacher} converges really fast and that the first term recovers the result by Hardy and Ramanujan \eqref{PartitionAsymptotics}. Here we want to find an exact formula of Rademacher type for the cubic partition function $a(n),$ by using a result of Zuckerman \cite{Zuckerman}, extending Rademacher's work and it can be seen as the apotheosis of the classical Circle Method. Using the Circle Method, Zuckerman computed exact formulas for Fourier coefficients of weakly holomorphic modular forms of arbitrary non-positive weight on subgroups of $\SL_2(\Z)$ with finite index in terms of the cusps of the underlying subgroup and the negative coefficients of the form at each cusp. We state the relevant results of Zuckerman in Section $2$ and prove the exact formula for $a(n)$ in Section $3.$ Similar to the case above we will obtain as an immediate consequence the asymptotics

$$a(n) \sim \frac{e^{\pi\sqrt{n-\frac{1}{8}}}}{8\left(n-\frac{1}{8}\right)^{\frac{5}{4}}}, n \rightarrow \infty.$$

Finally, we are going to prove a conjecture by Banerjee, Paule, Radu and Zeng \cite{BanerjeePauleRaduZeng} which predicts the following asymptotic formula for $\log a(n)$

\begin{equation*}
\log \left(a(n)\right) \sim \pi \sqrt{n} - \frac{5}{4} \log \left(n\right) - \log \left(8\right) - \left(\frac{15}{8\pi} + \frac{\pi}{16}\right)\frac{1}{\sqrt{n}}, \quad n \rightarrow \infty, \quad \frac{15}{8\pi} + \frac{\pi}{16} \approx 0.79
\end{equation*}

\section*{Acknowledgements}
The author wishes to thank Kathrin Bringmann and Walter Bridges for suggesting this problem, William Craig and Andreas Mono for sharing their knowledge on modular forms and helpful suggestions, and Johann Franke for helping me verify the results numerically. The author recieved funding from the European Research Council (ERC) under the European Union’s Horizon 2020 research and innovation programme (grant agreement No. 101001179).

\section{Zuckerman's result}
The following two sections follow Zuckerman's work \cite{Zuckerman} closely. Most of the results are completely taken over and translated into more modern language. His method, due to its generality involves numerous technical parameters which appear in the Cirlce Method and do not have obvious meaning without knowing the context they arise in. The interested reader should thus consult \cite{Zuckerman} for a more detailed account.

Let $\Gamma$ be a subgroup of $\SL_2(\Z)$ of finite index and let $F$ be a weakly holomorphic modular form of weight $k=-r, r>0.$  Thus, $F(\tau)$ satisfies a transformation equation of the form

\begin{equation*}
F\left(\frac{a\tau + b}{c\tau + d}\right) = \varepsilon(-i(c\tau + d))^kF(\tau), \quad \begin{pmatrix}
a &b \\
c & d
\end{pmatrix} \in \Gamma,
\end{equation*}

\noindent
where $\varepsilon = \varepsilon(a,b,c,d)$ lies on the unit circle and depends only on the transformation. If $c\neq 0,$ then $c$ is taken to be positive and we choose the branch of the argument such that 

$$ -\frac{\pi}{2} < \Arg(-i(c\tau + d)) < \frac{\pi}{2}.$$

\noindent
Since $\Gamma$ is of finite index in the modular group we can choose a complete finite system of inequivalent cusps of $\Gamma,$ which we will denote by $P_1, \dots P_s,$
where 

$$P_g = \frac{p_q}{q_g}, \quad \gcd(p_q,g_q) = 1, \quad g_q > 0, \quad  g=1, \cdots , s.$$

\noindent
In order to treat all our cusps symetrically, we will have to assume that the point at infinity does not belong to our chosen set of inequivalent cusps. This can always be achieved by considering a $\Gamma$-equivlant rational point instead. We will obtain a set of Fourier expansions $f_1, \cdots, f_s$ corresponding to our set of cusps in the variable $(\tau - P_g)^{-1}.$ Consider now any transformation $\gamma \in \SL_2(\Z),$ which must not necessarily belong to our subgroup $\Gamma.$ We then write

\begin{equation}\label{GeneralFourierExpansion}
F\left(\frac{a\tau + b}{c\tau + d}\right) = \varepsilon^{*}(-i(c\tau + d))^kF^{*}(\tau), \quad \gamma = \begin{pmatrix}
a &b \\
c & d
\end{pmatrix} \in \SL_2(\Z),
\end{equation}

\noindent
where $F^{*}(\tau)$ is now a usual Fourier expansion in $\tau,$ obtained from one of the expansions 

$$f_g(x) = \sum_{m =-\mu_g}^{\infty} a_m^{(g)} x^m \quad x = \exp\left(\frac{-2\pi i}{c_g(\tau - P_g)}\right),$$

\noindent
where $c_g>0$ will be specified below. Note that by choosing the identity matrix in \eqref{GeneralFourierExpansion} we get a usual Fourier expansion $F(\tau),$ which we will later want. Corresponding to each of the cusps $P_g$ we can find a transformation in $\Gamma$ 

\begin{equation*}
\tau' = \frac{a\tau + b}{c\tau + d},
\end{equation*}

\noindent
which can be written in the the form 

$$\frac{1}{\tau'-P_g} = \frac{1}{\tau - P_g} + c_g, \quad c_g >0.$$

\noindent
We can recover the transformation in the usual form as

\begin{equation*}
	\tau' = \frac{(c_gP_g +1) - c_gP_g^2}{c_g\tau + 1 - c_gP_g}.
\end{equation*}

\noindent
We then define $\alpha_g$ by the equation $\varepsilon e^{\frac{\pi i}{2}} = e^{-2\pi i \alpha_g},$ where $0\leq \alpha_g < 1$ and $\varepsilon$ is taken with respect to $\tau'.$

We now describe how to find $F^{*}(\tau)$ in terms of the Fourier expansions $f_g.$ We therefore choose any transformation $\gamma \in \SL_2(\Z).$ Then, $\frac{a}{c}$ is a rational point if $c\neq 0$ or the point at infinity if $c=0.$ In any case $\frac{a}{c}$ is $\Gamma$-equivalent to exactly one of our cusps $P_g.$ We can therefore find a transformation 

$$\frac{a_1P_g + b_1}{c_1P_g + d_1} = \frac{a}{c}, \quad c\geq 0, \quad \begin{pmatrix}
a_1 &b_1 \\
c_1 & d_1
\end{pmatrix} \in \Gamma,$$

\noindent
which takes $P_g$ into $\frac{a}{c}.$ We can then express $F^{*}(\tau)$ in terms of the Fourier expansion $f_g$ as follows

\begin{equation*}
F^{*}(\tau) = q_g^{k} \exp\left(\frac{2\pi i q_g^2}{c_g} \alpha_g \tau \right) \dot f_g\left(\exp\left(\kappa\frac{2\pi i q_g}{c_g} (a_1d-c_1b)\right) \exp\left(\frac{2 \pi i q_g^2}{c_g} \tau \right)\right),
\end{equation*}

\noindent
where $\kappa = \pm 1$ is taken as the solution to the following equations

$$p_q = \kappa(ad_1-cb_1), \quad q_g = \kappa(ca_1-ac_1).$$

\noindent
Furthermore, one obtains

$$\varepsilon*=\varepsilon(a_1,b_1,c_1,d_1) \exp\left(\kappa \frac{\pi i r}{2}\right) \exp\left(\kappa\frac{2\pi i q_g}{c_g} \alpha_g(a_1d-c_1b)\right).$$ This gives the Fourier expansion in spirit of \eqref{GeneralFourierExpansion} at any rational point $\frac{a}{c},$ including the point at infinity. When executing the circle method with these transformation equations there are numerous parameters that appear. First of all consider as usual the Farey fractions $\frac{h}{k}, \gcd(h,k) = 1, k>0, h\geq 0.$ Then, consider the point $P=P_g - \frac{k}{hc_g}$ which is either a rational point or the point at infinity. In any case it is $\Gamma$-equivalent to one of our cusps $P_\beta,$ where $\beta = \beta(h,k,g).$ There exists a transformation in $\Gamma$ which takes this point into $P_\beta$

\begin{equation*}
\frac{aP+b}{cP+d} = P_\beta, \quad \begin{pmatrix}
a &b \\
c & d
\end{pmatrix} \in \Gamma.
\end{equation*}

\noindent
Associated to this transformation is the parameter $\sigma_{h,k}^{(g)}$ implicitly defined as the solution of the equations

$$a\left(\frac{c_g}{q_g}p_gh-k\right) + bc_gh = \sigma_{h,k}^{(g)}p_\beta,$$

$$c\left(\frac{c_g}{q_g}p_gh-k\right) + dc_gh = \sigma_{h,k}^{(g)}q_\beta.$$

\noindent
It satisfies $-c_g \leq \sigma_{h,k}^{(g)} \leq c_g, \sigma_{h,k}^{(g)} \neq 0.$ The remaining technical parameters necessary to state Zuckerman's result are 

$$\delta_{h,k}^{(g)} = \begin{cases}
-1 & \text{if } \sigma_{h,k}^{(g)} > 0, \\
1 & \text{if } \sigma_{h,k}^{(g)} < 0,
\end{cases} \quad G_{h,k}^{(g)} = - \frac{2\pi}{kc_\beta}\sigma_{h,k}^{(g)}q_\beta(cP_g+d),$$
$$\Omega_{h,k} = \varepsilon^{-1}(a,b,c,d) \exp\left((1-\delta_{h,k})\frac{\pi i r}{2}\right) \exp\left(-\frac{2\pi i}{k} \left(\alpha_g h + \frac{\sigma_{h,k}^{(g)}q_\beta(cP_g+d)\alpha_\beta}{c\beta}\right)\right).$$

Finally, Zuckerman's result reads now as follows.

\begin{theorem}[Zuckerman's exact formula]\label{ZuckermanTheorem}
	Assume the notation and hypotheses above and denote by $I_r$ the Bessel function of order $r.$ If $n + \alpha_g > 0,$ then we have 
	
	\begin{multline*}
	a_n^{(g)} = \frac{2\pi c_g^{\frac{r-1}{2}}}{(\alpha_g + n)^{\frac{r+1}{2}}} \sum_{k=1}^{\infty}\sum_{\substack{0\leq h < k \\ (h,k) = 1}} \Omega_{h,k}^{(g)}\exp\bigg(-2\pi i n \frac{h}{k}\bigg)\\
	 \times \sum_{\nu = 1}^{\mu_\beta}  a_{-\nu}^{\beta} \exp\left(-\nu G_{h,k}^{(g)} i \right) \frac{\left|\sigma_{h,k}^{(g)}\right|q_\beta(\nu-\alpha_\beta)^{\frac{r+1}{2}}}{kc_\beta^{\frac{r+1}{2}}} \dot I_{r+1}\left(\frac{4\pi |\sigma_{h,k}^{(g)}|q_\beta\sqrt{(\nu-\alpha_\beta)(\alpha_g + n)}}{k\sqrt{c_gc_\beta}}\right).
	\end{multline*}
\end{theorem}

\section{Exact formula for the cubic partition function}

We recall that the cubic partition function was defined by the identity

\begin{equation*}
\sum_{n=0}^{\infty} a(n)q^n = \prod_{n=1}^{\infty} \frac{1}{(1-q^n)(1-q^{2n})},
\end{equation*}

\noindent
which is up to a power of $q$ equal to $(\eta(\tau)\eta(2\tau))^{-1}.$ Using the relation 

$$ 2\frac{a\tau + b}{c\tau + d} = \frac{a(2\tau) + 2b}{\frac{c}{2}(2\tau) + d}$$

\noindent
we can view $\eta(2\tau)$ in light of \eqref{DedekindEtaFunction} as a modular form of weight $\frac{1}{2}$ on the subgroup $\Gamma_0(2)$ of $\SL_2(\Z)$ with multiplier

$$\varepsilon = \varepsilon\left(a,2b,\frac{c}{2},d\right),$$

\noindent
Hence, we find that 

$$F(\tau) = \frac{1}{\eta(\tau)\eta(2\tau)}$$

\noindent
is a weakly modular form of weight $-1$ with multiplier and thus we can apply Zuckerman's result to find the Fourier coefficients of $F(\tau)$ and therefore a formula of $a(n).$ The group $\Gamma_0(2)$ has two inequivalent cusps $0$ and $\frac{1}{2}.$ We consider the transformations

$$ \tau' = \frac{\tau}{2\tau + 1}, \quad \tau' = \frac{3\tau - 1}{4\tau -1}$$

\noindent
which we write as

$$\frac{1}{\tau'} = \frac{1}{\tau} + 2, \quad \frac{1}{\tau'-\frac{1}{2}} = \frac{1}{\tau - \frac{1}{2}} + 4.$$

Thus, we find that $c_1 = 2$ and $c_2 = 4.$ Moreover we can now determine $\alpha_1$ and $\alpha_2.$ Since $F(\tau)$ has the multiplier

$$\varepsilon = \varepsilon(a,b,c,d)^{-1}\varepsilon\left(a,2b,\frac{c}{2},d\right)^{-1},$$

\noindent
a short calculation shows that $\alpha_1 = \alpha_2 = \frac{7}{8}.$ Since the point at infinity is $\Gamma_0(2)$-equivalent to $\frac{1}{2},$ we obtain that $F^{*}(\tau))$ is given in terms of the Fourier expansion $f_2.$ To find the exact expression for $F^{*}(\tau))$ we follow Zuckerman's method by choosing a transformation that maps $\frac{1}{2}$ to the point at infinity. Obvioulsy

$$ \begin{pmatrix}
-1 &0 \\
2 & -1
\end{pmatrix} \frac{1}{2} = i\infty.$$

\noindent
Hence, $\kappa = -1$ and 

\begin{equation*}
F^{*}(\tau)=\frac{1}{2}\exp\left(\frac{7\pi i}{8} \tau\right) f_1(-\exp(2\pi i \tau)) = \frac{1}{2}\exp\left(\frac{14\pi i}{8}\right) \sum_{m =-\mu_1}^{\infty} (-1)^m a_m^{(1)}.
\end{equation*}

\noindent
Furthermore, we consider the point $\frac{a}{c} = \frac{0}{1},$ which corresponds to the transformation  $\tau \mapsto -\frac{1}{\tau}.$ In this case the obvious equivalent point is $0$ itself and the identity is the corresponding transformation. Thus, in this case

$$F^{*}(\tau) = \exp\left(\frac{7\pi i}{8} \tau\right) f_2(\exp(\pi i \tau)).$$

We now determine $a_{-1}^{(g)}$ and $\mu_{-2}^{(g)}.$ In the first case we have on the one hand

$$F\left(\begin{pmatrix}
1 &0 \\
0 & 1
\end{pmatrix}\tau\right) = F(\tau) = \varepsilon^{*}(-1,0,2,-1) \frac{1}{2}\exp\left(\frac{14\pi i}{8}\tau\right) \sum_{m =-\mu_1}^{\infty} (-1)^m a_m^{(1)}$$

\noindent
with $\varepsilon^{*} = e^{\frac{7\pi i}{8}}.$ On the other hand we have by definition that 

\begin{equation}\label{FourierExpansion}
F(\tau) = e^{-\frac{\pi i \tau }{4}}\left( a(0) + a(1)e^{2\pi i \tau} + \dots \right) = e^{\frac{2\pi i 7}{8} \tau} \left( a(0)e^{-2\pi i \tau} + a(1) + \dots \right)
\end{equation}

\noindent
Comparing coefficients we see that $\mu_{-1}^{(2)} = 1 $ and $a_{-1}^{(2)} = -2e^{\frac{\pi i}{8}}.$ Furthermore, \eqref{FourierExpansion} implies that the Fourier coefficients are exactly $a(n)$ shifted down by one. Thus, we can recover $a(n)$ through the relation $n=m+1.$ A similar calculation shows that $\mu_{-1}^{(1)} = 1$ and $a_{-1}^{(2)} = \frac{e^{-\frac{\pi i}{2}}}{\sqrt{2}}.$

We now continue to determine the remaining parameters only for $g=2$ as the Fourier expansion at $\frac{1}{2}$ gives us the exact formula. For $\beta(h,k,2)$ it is required to find transformations that send $P_g-\frac{k}{c_gh}$ into $P_\beta.$ Let $h', h'', h'''$ be any solutions of 

$$ \left(\frac{k}{2} -h \right)h' \equiv 1 \pmod{2h}, \quad h' > 0, \quad \text{ for } k \equiv 0 \pmod{4},$$ 
$$ kh'' \equiv -1 \pmod{h}, \quad h'' > 0, \quad \text{ for } k \equiv 0 \pmod{2},$$ 
$$ (k-2h)h''' \equiv 1 \pmod{4h}, \quad h''' > 0, \quad \text{ for } k \equiv 1 \pmod{2}.$$ 

\noindent
Then, the following table gives the remaining parameters for needed to apply Theorem \ref{ZuckermanTheorem}

\begin{center}
\begin{tabular}{|cc|c|cc|cc|cc|c|c|c|cc|}
\hline \rule[-3mm]{0mm}{8mm}
$k$  && $a$   && $b$  && $c$ && $d$ & $\beta$ & $\sigma_{h,k}^{(2)}$ & $\delta_{h,k}^{(2)}$ && $G_{h,k}^{(2)}$   \\   \hline 

$ \equiv 1 \pmod{2}$ && $h'''$ && $\frac{(k-2h)h'''-1}{4h}$ && $2h''' + 4h$ && $\frac{(k-2h)h'''-1}{2h} - 2h + k$ & $2$ & $-1$ & $1$ && $\left(\frac{h'''k-1}{2kh} + 1\right)\pi$ \\

$ \equiv 0 \pmod{4}$ && $h'$ && $\frac{\left(\frac{k}{2}-h\right)h'- 1}{2h}$ && $2h' + 2h$ && $\frac{\left(\frac{k}{2}-h\right)h' - 1}{h} + \frac{k}{2} - h $ & $2$ & $-2$ & $1$ && $\left(\frac{h'k-2}{kh} + 1\right)\pi$ \\

$ \equiv 2 \pmod{4}$ && $h$ && $\frac{k-2h}{4}$ && $4h''$ && $\frac{kh'' + 1}{h} - 2h''$ & $1$ & $4$ & $-1$ && $-\left(\frac{4h''k+4}{kh}\right)\pi$ \\

\hline
\end{tabular}
\end{center}

Since the series in Theorem \ref{ZuckermanTheorem} converges absolutely, the terms can be rearranged and putting everything together, keeping in mind the shift $n = m+1,$ it follows that

\begin{align}
a(n) = &\frac{e^{-\frac{\pi i}{2}}}{2\sqrt{2}} \sum_{\substack{k \equiv 2 \pmod{4} \\ k > 0}} A_{k}(n+1) \frac{1}{n-\frac{1}{8}} I_2\left(\frac{2\pi\sqrt{n-\frac{1}{8}}}{k}\right) \nonumber \\
&- \frac{\pi e^{\frac{\pi i}{8}}}{2} \sum_{\substack{k \equiv 0 \pmod{4} \\ k > 0 }} \frac{1}{k} A_k(n+1) \frac{1}{n-\frac{1}{8}} I_2\left(\frac{\sqrt{2}\pi\sqrt{n-\frac{1}{8}}}{k}\right) \nonumber \\
&- \frac{\pi e^{\frac{\pi i}{8}}}{4} \sum_{\substack{k \equiv 1 \pmod{2} \\ k > 0 }} \frac{1}{k} A_k(n+1) \frac{1}{n-\frac{1}{8}} I_2\left(\frac{\sqrt{2}\pi\sqrt{n-\frac{1}{8}}}{2k}\right),
\end{align}

where 

\begin{multline}
	A_k(n) = \sum_{\substack{0\leq h < k \\ (h,k) = 1}} \varepsilon\left(h, \frac{k-2h}{4}, 4h'', \frac{kh''+1}{h} - 2h'' \right)^{-1} \exp(\pi i) \\
	\times \exp\left(-\frac{2\pi i 7}{8k}h + \frac{1}{8}\left(\frac{2h''k+2}{kh}\right)\pi i - 2\pi in\frac{h}{k}\right), \quad k \equiv 2 \pmod{4},
\end{multline}

\begin{multline}
A_k(n) = \sum_{\substack{0\leq h < k \\ (h,k) = 1}} \varepsilon\left(h', \frac{\left(\frac{k}{2}-h\right)h' - 1}{2h}, 2h'+2h', \frac{\left(\frac{k}{2} - h\right)h' - 1}{h} + \frac{k}{2} - h \right)^{-1} \exp(\pi i) \\
\times \exp\left(-\frac{2\pi i 7}{8k}h - \frac{1}{8}\left(\frac{h'k - 2}{kh} + 1\right)\pi i - 2\pi in\frac{h}{k}\right), \quad k \equiv 0 \pmod{4},
\end{multline}

\begin{multline}
A_k(n) = \sum_{\substack{0\leq h < k \\ (h,k) = 1}} \varepsilon\left(h''', \frac{(k-2h)h''' - 1}{4h}, 2h''' + 4h, \frac{(k-2h)h'''-1}{2h} - 2h + k \right)^{-1} \exp(\pi i) \\
\times \exp\left(-\frac{2\pi i 7}{8k}h - \frac{1}{8}\left(\frac{h'''k - 1}{2kh} + 1\right)\pi i - 2\pi in\frac{h}{k}\right), \quad k \equiv 2 \pmod{4},
\end{multline}

Changing the order of summation one last time we arrive at the following

\begin{theorem}
	For $n \geq 0$ we have the exact formula
	
	\begin{align}\label{ExactFormula}
	& a(n) = \frac{\pi e^{-\frac{\pi i}{2}}}{4\sqrt{2}\left(n-\frac{1}{8}\right)} \sum_{\substack{ l \equiv 1 \pmod{2} \\ l > 0}} \frac{1}{l} A_{l'}(n+1) I_2\left(\frac{\pi\sqrt{n-\frac{1}{8}}}{l}\right) \nonumber \\
	& - \frac{\pi e^{\frac{\pi i}{8}}}{2\left(n-\frac{1}{8}\right)} \sum_{\substack{ l \equiv 0 \pmod{2} \\ l > 0}} \frac{1}{l} A_{l'}(n+1) I_2\left(\frac{\sqrt{2}\pi \sqrt{n-\frac{1}{8}}}{l}\right),
	\end{align}
	
	\noindent
	where $l' = l$ if $l \equiv 1 \pmod{2},$ $l' = l$ if $l \equiv 0 \pmod{4},$ $l' = \frac{l}{2}$ if $l \equiv 2 \pmod 4.$
\end{theorem}

It is easy to see that the series converges rapidly and that the term $l=1$ dominates the rest of the sum. Thus, we may conclude that

$$a(n) \sim \frac{\pi}{4\sqrt{2}\left(n-\frac{1}{8}\right)} I_2\left(\pi\sqrt{n-\frac{1}{8}}\right).$$

\noindent
For the Bessel function of order $\nu \in \N$  the following asymptotics are known,

\begin{equation}
I_\nu(z) \sim \frac{e^{z}}{\sqrt{2\pi z}} \sum_{k=0}^{\infty} (-1)^k \frac{a_k(\nu)}{z^k},
\end{equation}

\noindent
see \cite{Temme}, where $a_k(\nu)$ is defined as

\begin{equation}
a_k(\nu) = (-1)^k\frac{\left(\frac{1}{2} - \nu \right)\left( \frac{1}{2} + \nu \right)}{2^k k!}, \quad k = 0,1,2,\dots
\end{equation}
\begin{equation}
(\lambda)_n = \frac{\Gamma(\lambda + n)}{\Gamma(\lambda)}
\end{equation}

\begin{equation}
(\lambda)_0 = 1
\end{equation}

\noindent
Thus, we find that the first order approximation is 

$$I_2(z) = \frac{e^{z}}{\sqrt{2\pi z}} \left( 1 + O\left(\frac{1}{z}\right)\right)$$

\noindent
We now obtain the following asymptotics.

\begin{corollary}
	We have
	\begin{equation}
	a(n) \sim \frac{e^{\pi\sqrt{n-\frac{1}{8}}}}{8\left(n-\frac{1}{8}\right)^{\frac{5}{4}}}
	\end{equation}
\end{corollary}

As a final consequence we prove a conjecture by Banerjee, Paule, Radu and Zeng.

\begin{corollary}
	We have
	$$\log a(n) \sim \pi \sqrt{n} - \frac{5}{4} \log n - \log 8 - \left(\frac{15}{8\pi} + \frac{\pi}{16}\right)\frac{1}{\sqrt{n}}, \quad n \rightarrow \infty.$$
\end{corollary}

\begin{proof}
	We need to shift our approximation from $n-\frac{1}{8}$ to $n.$ Thus, write
	
	$$e^{\pi\sqrt{n-\frac{1}{8}}} = e^{\pi\sqrt{n}}\left(1-\frac{\pi}{16\sqrt{n}} + O\left(\frac{1}{n}\right)\right),$$
	
	$$\frac{1}{8\left(n-\frac{1}{8}\right)^{\frac{5}{4}}} = \frac{1}{8n^{\frac{5}{4}}}\left(1+O\left(\frac{1}{n}\right)\right).$$
	
	\noindent
	Together with the second order approximation for the Bessel function 
	
	$$I_2(z) = \frac{e^{z}}{\sqrt{2\pi z}}\left(1-\frac{15}{8z} + O\left(\frac{1}{z^2}\right)\right)$$
	
	\noindent
	we obtain
	
	$$a(n) = \frac{e^{\pi\sqrt{n}}}{8n^{\frac{5}{4}}}\left(1-\frac{15}{8n} + O\left(\frac{1}{n^2}\right)\right)\left(1-\frac{\pi}{16\sqrt{n}} + O\left(\frac{1}{n}\right)\right)\left(1-\frac{\pi}{16\sqrt{n}} + O\left(\frac{1}{n}\right)\right)\left(1 + O\left(\frac{1}{n}\right)\right)$$
	
	\noindent
	Taking logarithms, we find as conjectured
	
	\begin{equation*}
	\log a(n) \sim \pi \sqrt{n} - \frac{5}{4} \log n - \log 8 - \left(\frac{15}{8\pi} + \frac{\pi}{16}\right)\frac{1}{\sqrt{n}}, \quad n \rightarrow \infty.
	\end{equation*}
	
\end{proof}

\begin{remark}
	We can refine our proof and obtain stronger asymptotics of the form 
	
	\begin{equation*}
	\log a(n) \sim \pi \sqrt{n} - \frac{5}{4} \log n - \log 8 - \left(\frac{15}{8\pi} + \frac{\pi}{16}\right)\frac{1}{\sqrt{n}} + \sum_{j=2}^{\infty} \frac{c_j}{n^{\frac{j}{2}}}
	\end{equation*}
	
	\noindent
	for some real $c_i$ by taking higher order expansions of the involved functions.
\end{remark}

\end{document}